\documentclass[12pt,reqno]{amsart}
\usepackage{amsmath}
\usepackage{amssymb}
\usepackage{amstext}
\usepackage{a4wide}
\usepackage{graphicx}
\allowdisplaybreaks \numberwithin{equation}{section}
\usepackage{color}

\numberwithin{equation}{section}

\newtheorem{theorem}{Theorem}[section]
\newtheorem{proposition}[theorem]{Proposition}

\newtheorem{lemma}[theorem]{Lemma}

\theoremstyle{definition}

\newtheorem{definition}[theorem]{Definition}

\theoremstyle{remark}
\newtheorem{remark}[theorem]{Remark}

\begin{document}

\title
{Steady vortex patches near a nontrivial irrotational flow}

 \author{Daomin Cao, Guodong Wang, Weicheng Zhan}

\address{Institute of Applied Mathematics, Chinese Academy of Science, Beijing 100190, and University of Chinese Academy of Sciences, Beijing 100049,  P.R. China}
\email{dmcao@amt.ac.cn}
\address{Institute of Applied Mathematics, Chinese Academy of Science, Beijing 100190, and University of Chinese Academy of Sciences, Beijing 100049,  P.R. China}
\email{wangguodong14@mails.ucas.ac.cn}
\address{Institute of Applied Mathematics, Chinese Academy of Science, Beijing 100190, and University of Chinese Academy of Sciences, Beijing 100049,  P.R. China}
\email{zhanweicheng16@mails.ucas.ac.cn}


\begin{abstract}
In this paper, we study the vortex patch problem in an ideal fluid in a planar bounded domain. By solving a certain minimization problem and studying the limiting behavior of the minimizer, we prove that for any harmonic function $q$ corresponding to a nontrivial irrotational flow, there exists a family of steady vortex patches approaching the set of extremum points of $q$ on the boundary of the domain. Furthermore, we show that each finite collection of strict extreme points of $q$ corresponds to a family of steady multiple vortex patches approaching it.
\end{abstract}

\maketitle

\section{Introduction}
In this paper, we study the incompressible steady flow in a planar bounded domain. The motion is governed by the following Euler equations

\begin{equation}\label{1-1}
\begin{cases}

(\mathbf{v}\cdot\nabla)\mathbf{v}=-\nabla P,\\
 \nabla\cdot\mathbf{v}=0,
\end{cases}
\end{equation}
where $\mathbf{v}=(v_1,v_2)$ is the velocity field and $P$ is the scalar pressure. The vorticity of the flow is defined by \[\omega=curl\mathbf{v}:=\partial_1v_2-\partial_2v_1.\]
If $\omega\equiv0$, then the flow is said to be irrotational.

Irrotational flows can be completely classified. Let $\mathbf{v}$ be a solution of \eqref{1-1}, then by the divergence-free condition $\nabla\cdot \mathbf{v}=0$ and Green's theorem, we have
\[\mathbf{v}=\nabla^\perp \psi:=(\partial_2\psi,-\partial_1\psi)\]
 for some function $\psi$, which we call the stream function. It is easy to check that
 \[-\Delta\psi=\omega.\]
 So the flow is irrotational if and only if $\psi$ is harmonic. Conversely, it is easy to prove that for any harmonic function $\psi$, there exists an irrotational flow with $\psi$ as its stream function. In this sense, an irrotational flow is equivalent to a harmonic function.

In this paper, we are concerned with steady flows with nonvanishing vorticity. More precisely, we prove that for any given nontrivial irrotational flow(the velocity field is not zero), there exists a family of steady vortex patch solutions near this flow. Here by vortex patch we mean that the vorticity $\omega$ has the form $\omega=\kappa I_A$, where $\kappa$ is a real number representing the vorticity strength, $A$ is a measurable set, and $I_A$ denotes the characteristic function of $A$, namely, $I_A=1$ in $A$ and $I_A=0$ elsewhere.

Vortex patches are a special class of non-smooth solutions of the two-dimensional Euler equations appropriate for modeling an isolated region of constant vorticity. In the past few decades, the extensive study of vortex patches has led to many interesting and significant results. In this paper, we focus on the construction of steady vortex patches. There exist a great literatures dealing with this problem; see for example \cite{CPY,CW,CW2,HM,T,W} and the references listed therein. An efficient method to study the vortex patch problem is the vorticity method. It was first established by Arnold and Khesin \cite{A,AK} and later developed by many authors \cite{Ba,B,B2,EM,FB,T}. Roughly speaking, the vorticity method asserts that a steady flow is in fact a constrained critical point of the kinetic energy, and the stability of the flow is equivalent to the nondegeneracy of that critical point. By maximizing the kinetic energy in a weakly closed subset in $L^\infty$ that contains the class of isovortical patches and studying the limiting behavior, Turkington \cite{T} constructed a family of steady vortex patches concentrating at a global minimum point of the Robin function of the domain. Later Burton \cite{B,B2} considered the maximization of the kinetic energy on general rearrangement classes, and by which he found more dynamically possible equilibria of planar vortex flows. Another method to study the vortex patch problem is developed by Cao et al. \cite{CPY}. By using a reduction argument for the stream function, they obtained steady multiple vortex patches near any given nondegenerate critical point of the Kirchhoff-Routh function. They also proved the uniqueness in \cite{CGPY} under certain assumptions.

Most of the previous results in \cite{CPY,HM,T,W} were concerned with the desingularization of point vortices. According to the vortex model, the evolution of a finite number of concentrated vortices in two dimensions is described by a dynamical system involving the Kirchhoff-Routh function; See \cite{L} for a general discussion. A natural question is the connection between the vortex model and the Euler equations. More specifically, for any critical point of the Kirchhoff-Routh function, can we construct a family of steady solutions of the Euler equations concentrating near that point? The answer is yes in some situations, especially when that critical point is nondegenerate; See \cite{CPY}. When one deals with desingularization of point vortices, the vorticity converges to a Dirac measure, in which case the vorticity amount is usually fixed and the vorticity strength goes to infinity.
As a contrast, our results in this paper are essentially of perturbation type. By using an adaption of the method in \cite{T}, we consider a certain variational problem in which the vorticity strength is fixed and the vorticity amount goes to zero. In this situation, the limiting behavior is mostly determined by the background irrotational flow, rather than the Kirchhoff-Routh function. Finally by deriving asymptotic estimate for the Lagrange multiplier, we are able to show that the support of vorticity shrinks to the boundary of the domain.

 It is also worth mentioning that in \cite{LYY} the authors obtained a similar existence result for vorticity without jump by considering a semilinear elliptic equation satisfied by the corresponding stream function. For the three dimensional case, steady flows with nonvanishing vorticity near an irrotational flow can also be constructed in some special cases; See \cite{Al,TX} for instance.

This paper is organized as follows. We first give a description of our problem and state the main results in Section 2. Then in Section 3 and 4 we solve a minimization problem and study the limiting behavior respectively to prove the main results. In Section 5 we briefly discuss the maximization case and obtain two similar results of existence.

\section{Main results}

Let $D\subset\mathbb{R}^2$ be a bounded and simply connected domain with a smooth boundary, $\partial D$.
Let $q$ be a harmonic function in $D$ corresponding to an irrotational flow $(\mathbf{v}_0,P_0)$ with $\mathbf{v}_0=\nabla^\perp q:=(\partial_2q,-\partial_1q)$. Then we have
\begin{equation}\label{1-2}
\begin{cases}
(\mathbf{v}_0\cdot\nabla)\mathbf{v}_0=-\nabla P_0 &\text{in }D,\\
 \nabla\cdot\mathbf{v}_0=0 &\text{in }D,\\
 \mathbf{v}_0\cdot\mathbf{n}= -\frac{\partial q}{\partial\mathbf{\nu}} &\text{on }\partial D,
\end{cases}
\end{equation}
where $\mathbf{n}=(n_1,n_2)$ is the exterior unit normal to the boundary $\partial D$, and $\mathbf{\nu}=\mathbf{n}^\perp:=(n_2,-n_1)$ denotes clockwise rotation through $\frac{\pi}{2}$ of $\mathbf{n}$.

To find a solution with nonvanishing vorticity near $\mathbf{v}_0$, we consider the following Euler equations with the same boundary condition as \eqref{1-2}
\begin{equation}\label{1-3}
\begin{cases}

(\mathbf{v}\cdot\nabla)\mathbf{v}=-\nabla P &\text{in }D,\\
 \nabla\cdot\mathbf{v}=0 &\text{in }D,\\
 \mathbf{v}\cdot\mathbf{n}= -\frac{\partial q}{\partial\mathbf{\nu}} &\text{on }\partial D.
\end{cases}
\end{equation}

Now we simplify \eqref{1-3} by using its vorticity formulation. Set $\omega=curl\mathbf{v}$. Taking the curl in the first equation in \eqref{1-3} we get
\begin{equation}\label{1-4}
   \nabla\cdot(\omega\mathbf{v})=0.
\end{equation}
On the other hand, we can recover $\mathbf{v}$ from $\omega$ in the following way. Since
\begin{equation}\label{1-5}
\begin{cases}
\nabla\cdot(\mathbf{v}-\mathbf{v}_0)=0 &\text{in } D,\\
(\mathbf{v}-\mathbf{v}_0)\cdot\mathbf{n}=0 &\text{on } \partial D,
\end{cases}
\end{equation}
we obtain $\mathbf{v}-\mathbf{v}_0=\nabla^\perp\psi$ for some function $\psi$ with $\psi=$constant on $\partial D$. Since $D$ is simply connected, without loss of generality, we assume that $\psi=0$ on $\partial D$ by adding a suitable constant. Therefore $\psi$ can be uniquely determined by $\omega$
\begin{equation}\label{1-6}
   \begin{cases}
-\Delta\psi=\omega &\text{in }D,\\
\psi=0 &\text{on } \partial D.
\end{cases}
\end{equation}
Set $G(\cdot,\cdot)$ to be the Green's function for $-\Delta$ in $D$ with zero Dirichlet data on $\partial D$. Then $\psi$ can be expressed in terms of the Green's operator as follows
\begin{equation}\label{1-8}
   \psi(x)=G\omega(x):=\int_DG(x,y)\omega(y)dxdy.
\end{equation}

In other words, in order to solve \eqref{1-3}, it suffices to consider the following equation satisfied by $\omega$
\begin{equation}\label{1-7}
   \nabla\cdot(\omega\nabla^\perp(G\omega+q))=0.
\end{equation}

Since we are going to deal vortex patches which are discontinuous, it is necessary to interpret $\eqref{1-7}$ in the weak sense, namely, we need to introduce the notation of weak solution to \eqref{1-7}.

\begin{definition}\label{1-9}
We call $\omega\in L^\infty(D)$ a weak solution of \eqref{1-7} if

\begin{equation}\label{1-10}
    \int_D\omega\nabla^\perp(G\omega+q)\cdot\nabla\xi dx=0
\end{equation}
for any $\xi\in C^\infty_c(D)$.
\end{definition}

It should be noted that if $\omega\in L^\infty(D)$, then by $L^p$ estimate and Sobolev embedding $G\omega\in C^{1,\alpha}(D)$ for any $\alpha\in(0,1)$, therefore the integral in \eqref{1-10} makes sense.

Our first result is as follows.
\begin{theorem}\label{1-11}
Let $q\in C^2(D)\cap C^1(\overline{D})$ be a harmonic function and $\kappa$ be a positive real number. Set $\mathcal{S}:=\{x\in \overline{D}\mid q(x)=\min_{ \overline{D}}q\}$. Then
for any given positive number $\lambda$ with $\lambda<\kappa|D|$($|\cdot|$ denotes the two-dimensional Lebesgue measure), there exists a weak solution $\omega^\lambda$ of \eqref{1-7} having the form
\begin{equation}\label{1-12}
  \omega^\lambda=\kappa I_{\Omega^\lambda}, \,\, \Omega^\lambda=\{x\in D\mid G\omega^\lambda(x)+q(x)<\mu^\lambda\},\,\,\kappa|\Omega^\lambda|=\lambda
\end{equation}
for some $\mu^\lambda\in\mathbb{R}$ depending on $\lambda$.
Furthermore, if $q$ is not a constant, then $\mathcal{S}\subset\partial D$ and $\Omega^\lambda$ approaches $\mathcal{S}$ as $\lambda\rightarrow0$, or equivalently, for any $\delta>0$, there exists $\lambda_0>0$, such that for any $\lambda<\lambda_0$, we have
\begin{equation}\label{1-14}
\Omega^\lambda\subset\mathcal{S}_\delta:=\{x\in D\mid dist(x,\mathcal{S})<\delta\}.
\end{equation}
\end{theorem}

\begin{remark}\label{2-101}
In Theorem \ref{1-11}, if $q$ is not a constant, then by the strong maximum principle
\[\{x\in \overline{D}\mid q(x)=\min_{ \overline{D}}q\}\subset \partial D,\]
so $\Omega^\lambda$ approaches the boundary of the domain as $\lambda\rightarrow0^+.$
\end{remark}

\begin{remark}\label{2-102}
In Theorem \ref{1-11} we show the existence of steady vortex patches with positive vorticity near the set of global minimum points of $q$. Since $(\omega,q)$ satisfies \eqref{1-7} if and only if $(-\omega,-q)$ satisfies \eqref{1-7}, so by reverting the signs of $\omega^\lambda$ and $q$ in Theorem \ref{1-11}, we can also prove the existence of steady vortex patches with negative vorticity near the set of global maximum points of $q$. In Section 5 we consider the other two cases: vortex patches with positive vorticity near the set of global maximum points of $q$ and vortex patches with negative vorticity near the set of global minimum points of $q$.
\end{remark}

Our next result shows that each finite collection of strict extreme points of $q$ corresponds to a family of steady multiple vortex patches shrinking to it.
\begin{theorem}\label{1-15}
Let $q\in C^2(D)\cap C^1(\overline{D})$ be a harmonic function, $k,l$ be two nonnegative integers and $\kappa_1,\cdot\cdot\cdot,\kappa_{k+l}$ be $k+l$ positive real numbers. Suppose that $\{x_1,x_2,\cdot\cdot\cdot,x_k\}\subset\partial D$ are $k$ different strict local minimum points of $q$ on $\overline{D}$, and $\{x_{k+1},x_{k+2},\cdot\cdot\cdot,x_{k+l}\}\subset\partial D$ are $l$ different strict local maximum points of $q$ on $\overline{D}$.
Then there exists a $\lambda_0>0$, such that for any $0<\lambda<\lambda_0$, there exists a weak solution $w^\lambda$ of \eqref{1-7} having the form
\begin{equation}\label{1-16}
  w^\lambda=\sum_{i=1}^k\kappa_iI_{U_i^\lambda}-\sum_{j=k+1}^{k+l}\kappa_jI_{U_j^\lambda},
\end{equation}
where for each $1\leq i\leq k$
\begin{equation}\label{5-4-3}
U_i^\lambda =\{x\in D\mid Gw^\lambda(x)+q(x)<\nu_i^\lambda\}\cap{B_{\delta_0}(x_i)},\,\,|U_i^\lambda|=\lambda,
\end{equation}
and for each $k+1\leq j\leq k+l$
\begin{equation}\label{5-5-3}
U_j^\lambda =\{x\in D\mid Gw^\lambda(x)+q(x)>\nu_j^\lambda\}\cap{B_{\delta_0}(x_j)},\,\,|U_j^\lambda|=\lambda,
\end{equation}
where $\nu_p^\lambda\in\mathbb{R},$ for $1\leq p\leq k+l$, is the Lagrange multiplier depending on $\lambda$.
Here $\delta_0$ is chosen to be sufficiently small such that $x_i$ is the unique minimum point of $q$ on $\overline{B_{\delta_0}(x_i)\cap D}$ for $i=1,\cdot\cdot\cdot,k$, $x_i$ is the unique maximum point of $q$ on $\overline{B_{\delta_0}(x_j)\cap D}$ for $j=k+1,\cdot\cdot\cdot,k+l$, and $\overline{B_{\delta_0}(x_{p_1})\cap D}\cap\overline{B_{\delta_0}(x_{p_2})\cap D}=\varnothing$ for $1\leq p_1,p_2\leq k+l, p_1\neq p_2.$
Moreover, $U_p^\lambda$ shrinks to $x_p$ for each $1\leq p\leq k+l$ as $\lambda\rightarrow0^+$, or equivalently, for any $\delta>0$, there exists a $\lambda_0>0$, such that for any $0<\lambda<\lambda_0$, we have
\begin{equation}\label{1-17}
U^\lambda_p\subset B_{\delta}(x_p)\cap D.
\end{equation}

\end{theorem}

\section{Proof of Theorem \ref{1-11}}
In this section we prove Theorem \ref{1-11}. To begin with, we consider a minimization problem for the vorticity and study the limiting behavior of the minimizer.

\subsection{Minimization Problem}

Let $\kappa$ be a fixed positive number. For $0<\lambda<\kappa|D|$, we define the vorticity class $\mathcal{M}^\lambda$ as follows
\begin{equation}\label{3-1}
\mathcal{M}^\lambda:=\{\omega\in L^\infty(D)\mid 0\leq\omega\leq\kappa, \int_D\omega(x)dx=\lambda\}.
\end{equation}
Note that $\mathcal{M}^\lambda$ is not empty since $\lambda<\kappa|D|$. The variational problem is to minimize the functional
\begin{equation}\label{3-2}
  E(\omega):=\frac{1}{2}\int_D\int_DG(x,y)\omega(x)\omega(y)dxdy+\int_Dq(x)\omega(x)dx
\end{equation}
in the class $\mathcal{M}^\lambda$, that is,
\begin{equation}
c_\lambda=\inf\{E(\omega)\mid\omega\in\mathcal{M}^\lambda\}.
\end{equation}

We call $\upsilon\in\mathcal{M}^\lambda$ an absolute minimizer if $E(\upsilon)\leq E(\omega)$ for all $\omega\in \mathcal{M}^\lambda$. In order to study the existence of an absolute minimizer of $E$, we need to establish two preliminary results.

\begin{lemma}
$\mathcal{M}^\lambda$ is a sequentially compact subset of $L^\infty(D)$ in the weak star topology, that is, for any sequence $\{\omega_n\}\subset L^\infty(D)$, $n=1,2,\cdot\cdot\cdot$, there exist a subsequence $\{\omega_{n_j}\}$ and $w_0\in\mathcal{M}^\lambda$ such that for any $\zeta\in L^1(D)$
\begin{equation}\label{3-344}
\lim_{j\rightarrow+\infty}\int_D\zeta(x)\omega_{n_j}(x)dx=\int_D\zeta(x)\omega_{0}(x)dx.
\end{equation}

\end{lemma}
\begin{proof}
As $\mathcal{M}^\lambda$ is clearly a bounded set in $L^\infty(D)$, there is, for any sequence $\{\omega_n\}\subset L^\infty(D)$, a subsequence $\{\omega_{n_j}\}$ such that $\omega_{n_j} \rightarrow \omega_0$ weakly star in $L^\infty$ as $j\rightarrow+\infty$ for some $\omega_0\in L^\infty(D)$. So it suffices to show $\omega_0\in\mathcal{M}^\lambda,$ namely, $\int_D\omega_0(x)dx=1$ and $0\leq\omega_0\leq \kappa$ a.e. in $D$.

Firstly, by choosing $\zeta\equiv1$ in \eqref{3-344} we have
\[\lim_{j\rightarrow +\infty}\int_D\omega_{n_j}(x)dx=\int_D\omega_0(x)dx=1.\]
Now we prove $ \omega_0\leq\kappa$ by contradiction. Suppose that $|\{x\in D\mid\omega_0(x)>\kappa\}|>0$, then there exists a $\varepsilon_0>0$ such that $|\{x\in D\mid\omega_0(x)\geq\kappa+\varepsilon_0\}|>0$. Denote $A=\{x\in D\mid\omega_0(x)\geq\kappa+\varepsilon_0\}$. Then by taking $\zeta=I_A$ in \eqref{3-344} we have
\[0=\lim_{j\rightarrow +\infty}\int_D(\omega_0(x)-\omega_{n_j}(x))\zeta(x)dx=\lim_{j\rightarrow +\infty}\int_{A}(\omega_0(x)-\omega_{n_j}(x))dx\geq\varepsilon_0|A|>0,\]
which is a contradiction. Similarly we can prove $\omega_0\geq 0$.

\end{proof}

\begin{lemma}\label{3-345}
$E$ is weakly star continuous in $L^\infty(D)$.
\end{lemma}
\begin{proof}
Let $\{f_n\}$ be a sequence in $L^\infty(D)$ such that for some $f_0\in L^\infty(D)$, $f_n\rightarrow f_0$ weakly star in $ L^\infty(D)$ as $n\rightarrow+\infty$. It suffices to prove that $\lim_{n\rightarrow+\infty}E(f_n)=E(f_0)$. First by $L^p$ estimate we have $Gf_n\rightarrow Gf_0$ weakly in $W^{2,p}$ for any $1<p<+\infty,$ then by Sobolev embedding
$Gf_n\rightarrow Gf_0$ in $C^1(\overline{D})$, as $n\rightarrow+\infty$. By the definition of weak star convergence we can easily deduce that $\lim_{n\rightarrow+\infty}E(f_n)=E(f_0)$.

\end{proof}

Now we are ready to show the existence of an absolute minimizer.
\begin{proposition}\label{3-3}
$c_\lambda$ is achieved by an absolute minimizer $\omega^\lambda\in\mathcal{M}^\lambda$ with the form
\begin{equation}\label{3-4}
  \omega^\lambda=\kappa I_{\Omega^\lambda},\,\,\Omega^\lambda=\{x\in D\mid G\omega^\lambda(x)+q(x)<\mu^\lambda\}
\end{equation}
for some $\mu^\lambda$ depending on $\lambda$.
\end{proposition}

\begin{proof}
First we show that $E$ attains its minimum on $\mathcal{M}^\lambda$. By integration by parts, we have for any $\omega\in \mathcal{M}^\lambda$
\begin{equation}\label{3-5}
 E(\omega)=\frac{1}{2}\int_D|\nabla G\omega(x)|^2dx+\int_Dq(x)\omega(x)dx\geq \lambda\min_{\overline{D}}q,
\end{equation}
which means that $E$ is bounded from below on $\mathcal{M}^\lambda$. Now we choose a sequence $\{\omega_n\}\subset\mathcal{M}^\lambda$ such that
\[\lim_{n\rightarrow+\infty}E(\omega_n)=\inf_{\omega\in\mathcal{M}^\lambda}E(\omega).\]
Since $\mathcal{M}^\lambda$ is a sequentially compact subset of $L^\infty(D)$ in the weak star topology, there exists a subsequence $\{\omega_{n_j}\}$ and a $\omega^\lambda\in \mathcal{M}^\lambda$ such that $\omega_{n_j}\rightarrow \omega^\lambda$ weakly star in $L^\infty(D)$ as $j\rightarrow+\infty$. Then by the weak star continuity of $E$ in $L^\infty(D)$ we obtain
\begin{equation}\label{3-6}
 E(\omega^\lambda)=\lim_{j\rightarrow+\infty}E(\omega_{n_j})=\inf_{\omega\in\mathcal{M}^\lambda}E(\omega),
\end{equation}
which means that $\omega^\lambda$ is an absolute minimizer.

Now we prove that $\omega^\lambda$ has the following form
\begin{equation}\label{3-7}
  \omega^\lambda=\kappa I_{\Omega^\lambda},\,\,\Omega^\lambda=\{x\in D\mid G\omega^\lambda(x)+q(x)<\mu^\lambda\}
\end{equation}
for some $\mu^\lambda$ depending on $\lambda$. To show this, we choose a family of test functions $\omega_s=\omega^\lambda+s(z_0-z_1)$, $s>0$, where $z_0$ and $z_1$ satisfy
\begin{equation}
\begin{cases}
z_0,z_1\in L^\infty(D),\,z_0,z_1\geq 0,\text{ a.e. in } D,

 \\ \int_Dz_0(x)dx=\int_D z_1(x)dx,
 \\z_0=0 \quad\text{in } D\setminus\{x\in D\mid\omega^\lambda(x)\leq\kappa-\delta\},
 \\z_1=0 \quad\text{in } D\setminus\{x\in D\mid\omega^\lambda(x)\geq\delta\}.
\end{cases}
\end{equation}
Here $\delta$ is a positive number. It is not hard to check that for fixed $z_0,z_1$ and $\delta$, $\omega_s\in \mathcal{M}^\lambda$ provided that $s$ is sufficiently small. Since $\omega^\lambda$ is an absolute minimizer, we have
\[0\leq\frac{dE(\omega_s)}{ds}|_{s=0^+}=\int_Dz_0(x)(G\omega^\lambda(x)+q(x))dx-\int_Dz_1(x)(G\omega^\lambda(x)+q(x))dx.\]
By the choice of $z_0$ and $z_1$ we obtain
\begin{equation}\label{1-103}
\sup_{\{x\in D\mid\omega^\lambda(x)>0\}}(G\omega^\lambda+q)\leq\inf_{\{x\in D\mid\omega^\lambda(x)<\kappa\}}(G\omega^\lambda+q).
\end{equation}
Since $D$ is simply-connected and $G\omega^\lambda+q$ is continuous, \eqref{1-103} is in fact an equality, that is,
\begin{equation}
\sup_{\{x\in D\mid\omega^\lambda(x)>0\}}(G\omega^\lambda+q)=\inf_{\{x\in D\mid\omega^\lambda(x)<\kappa\}}(G\omega^\lambda+q).
\end{equation}
Now we can define
\[\mu^\lambda:=\sup_{\{x\in D\mid\omega^\lambda(x)>0\}}(G\omega^\lambda+q)=\inf_{\{x\in D\mid\omega^\lambda(x)<\kappa\}}(G\omega^\lambda+q).\]
It is clear that
\begin{equation}
\begin{cases}
\omega^\lambda=0\text{\,\,\,\,\,\,a.e.\,} \text{in }\{x\in D\mid G\omega^\lambda(x)+q(x)>\mu^\lambda\},
 \\ \omega^\lambda=\kappa\text{\,\,\,\,\,\,a.e.\,} \text{in }\{x\in D\mid G\omega^\lambda(x)+q(x)<\mu^\lambda\}.
\end{cases}
\end{equation}
By the property of Sobolev functions, on the level set $\{x\in D\mid G\omega^\lambda(x)+q(x)=\mu^\lambda\}$, we have $\nabla (G\omega^\lambda+q)=0\text{\,\,a.e.}$, therefore $\omega^\lambda=-\Delta( G\omega^\lambda+q)=0\text{\,\,a.e.}$.

Altogether, we obtain
\[\omega^\lambda=\kappa I_{\{x\in D\mid G\omega^\lambda(x)+q(x)<\mu^\lambda\}},\]
which is the desired result.

\end{proof}

Turkington \cite{T} considered the maximization of $E$ on $\mathcal{M}^\lambda$. In that case, the maximizer may not be unique, especially for domains with certain symmetries. But here for the minimization problem we have

\begin{lemma}\label{3-1001}
There is a unique absolute minimizer of $E$ on $\mathcal{M}^\lambda$.

\end{lemma}
 \begin{proof}
 Suppose there are two minimizers $\omega^\lambda_1,\omega^\lambda_2$ of $E$ on $\mathcal{M}^\lambda$. By Proposition \ref{3-3}, $\omega^\lambda_i=\kappa I_{\Omega^\lambda_i}$ for $i=1,2$. Now we define a new function $\omega^\lambda_3=\frac{1}{2}(\omega^\lambda_1+\omega^\lambda_2),$ or equivalently,
\begin{equation}\label{3-8}
\omega^\lambda_3(x)=
\begin{cases}
\frac{1}{2}\kappa &x\in \Omega^\lambda_1\bigtriangleup\Omega^\lambda_2,\\
\kappa &x\in\Omega^\lambda_1\cap\Omega^\lambda_2,\\
0 &x\in(\Omega^\lambda_1\cup\Omega^\lambda_2)^c.
\end{cases}
\end{equation}
It is easy to see that $\omega^\lambda_3\in\mathcal{M}^\lambda$, and $\omega^\lambda_1=\omega^\lambda_2$ if and only if $\omega^\lambda_3$ has the form $\omega^\lambda_3=\kappa I_{\Omega^\lambda_3}$ for some $\Omega^\lambda_3\subset D$. Now we calculate $E(\omega^\lambda_3),$

\begin{equation}\label{3-9}
  \begin{split}
  E(\omega^\lambda_3)&=\frac{1}{8}\int_D\int_DG(x,y)(\omega^\lambda_1+\omega^\lambda_2)(x)(\omega^\lambda_1+\omega^\lambda_2)(y)dxdy  +\frac{1}{2}\int_Dq(x)(\omega^\lambda_1+\omega^\lambda_2)(x)dx\\
  &=\frac{1}{2}E(\omega^\lambda_1)+\frac{1}{4}\int_D\int_DG(x,y)\omega^\lambda_1(x)\omega^\lambda_2(y)dxdy
  +\frac{1}{4}\int_Dq(x)(\omega^\lambda_1+\omega^\lambda_2)(x)dx\\
  &\leq \frac{1}{2}E(\omega^\lambda_1)+\frac{1}{8}\int_D\int_DG(x,y)\omega^\lambda_1(x)\omega^\lambda_1(y)dxdy
  +\frac{1}{8}\int_D\int_DG(x,y)\omega^\lambda_2(x)\omega^\lambda_2(y)dxdy\\
  &+\frac{1}{4}\int_Dq(x)(\omega^\lambda_1+\omega^\lambda_2)(x)dx\\
  &=E(\omega^\lambda_1).
  \end{split}
\end{equation}
Here we use the fact that $\int_D\int_DG(x,y)(\omega^\lambda(x)-\omega^\lambda_2(x))(\omega^\lambda_1(y)-\omega^\lambda_2(y))dxdy\geq0$. By \eqref{3-9} we know that
$\omega^\lambda_3$ is also a minimizer of $E$, then again by Proposition \ref{3-3} $\omega^\lambda_3=\kappa I_{\Omega^\lambda_3}$ for some $\Omega^\lambda_3\subset D$, which implies $\omega^\lambda_1=\omega^\lambda_2.$

\end{proof}

\begin{remark}
$\mathcal{M}^\lambda$ is in fact a convex subset of $L^\infty(D)$ and $E$ is a strictly convex functional on $\mathcal{M}^\lambda$.
\end{remark}

\subsection{Limiting behavior of the minimizer}
Now we analyze the limiting behavior of the minimizer $\omega^\lambda$ obtained in Proposition \ref{3-3} as $\lambda\rightarrow 0^+$, which will also be used in the next subsection. For simplicity, we will use $C$ to denote various positive numbers not depending on $\lambda$.

\begin{lemma}\label{3-10}
We have the following upper bound for $E(\omega^\lambda)$
\begin{equation}
E(\omega^\lambda)\leq \lambda\min_{\overline{D}}q+C\lambda^{\frac{3}{2}}.
\end{equation}
\end{lemma}

\begin{proof}
The basic idea is to choose a suitable test function. Let $x_0\in\partial D$ be a minimum point of $q$ on $\overline{D}$. Since $\partial D$ is smooth, $D$ satisfies the interior sphere condition at $x_0\in \partial D$. Therefore for $\lambda$ sufficiently small we can choose a ball $B_\varepsilon(x^\lambda)\subset D$ with $|x^\lambda-x_0|=\varepsilon$, where $\varepsilon$ satisfies $\kappa\pi\varepsilon^2=\lambda$. Now we define the test function to be $\upsilon^\lambda=\kappa I_{B_\varepsilon(x^\lambda)}.$ It is obvious that $\upsilon^\lambda\in \mathcal{M}^\lambda$, and
\begin{equation}\label{3-11}
\begin{split}
  E(\omega^\lambda)&\leq E(\upsilon^\lambda)=\frac{1}{2}\kappa\int_{B_\varepsilon(x^\lambda)}G\upsilon(x)dx+\kappa\int_{B_\varepsilon(x^\lambda)}q(x)dx.
  \end{split}
\end{equation}
To estimate the first integral in \eqref{3-11}, we use $L^p$ estimate and Sobolev embedding to obtain
\begin{equation}\label{3-12}
|G\upsilon|_{L^\infty(D)}\leq C|G\upsilon|_{W^{2,2}(D)}\leq C|\upsilon|_{L^2(D)}=C\lambda^\frac{1}{2}.
\end{equation}
To estimate the second integral in \eqref{3-11}, we observe that for any $x\in B_\varepsilon(x^\lambda)$
\begin{equation}\label{3-13}
|q(x)-q(x_0)|\leq |\nabla q|_{L^\infty(D)}|x-x_0|\leq C\varepsilon.
\end{equation}
By combining \eqref{3-11}, \eqref{3-12} and \eqref{3-13} together we get
\begin{equation}\label{3-14}
E(\omega^\lambda)\leq \lambda\min_{\overline{D}}q+C\lambda^{\frac{3}{2}}.
\end{equation}

\end{proof}

\begin{lemma}\label{3-15}
The following estimate holds
\begin{equation}\label{3-16}
\int_D(G\omega^\lambda(x)+q(x)-\mu^\lambda)\omega^\lambda(x)dx\geq -C\lambda^{\frac{3}{2}}.
\end{equation}

\end{lemma}

\begin{proof}
For convenience we denote $\varphi^\lambda=G\omega^\lambda+q-\mu^\lambda$ and $\varphi^\lambda_-=\min{\{\varphi,0\}}$. First by H\"older's inequality
\begin{equation}\label{3-17}
\int_D(G\omega^\lambda(x)+q(x)-\mu^\lambda)\omega^\lambda(x)dx=\kappa\int_{\Omega^\lambda}\varphi(x)dx\geq
-\kappa|\Omega^\lambda|^{\frac{1}{2}}(\int_{\Omega^\lambda}|\varphi^\lambda(x)|^2dx)^{\frac{1}{2}}.
\end{equation}

On the other hand, by the Sobolev embedding $W^{1,1}(D)\hookrightarrow L^2(D)$ and H\"older's inequality
\begin{equation}\label{3-18}
\begin{split}
   (\int_{\Omega^\lambda}|\varphi^\lambda(x)|^2dx)^{\frac{1}{2}}&= (\int_{D}|\varphi_-^\lambda(x)|^2dx)^{\frac{1}{2}}\\
   &\leq C(\int_D|\varphi^\lambda_-(x)|dx+\int_D|\nabla\varphi^\lambda_-(x)|dx)\\
   &=C(\int_{\Omega^\lambda}|\varphi^\lambda(x)|dx+\int_{\Omega^\lambda}|\nabla\varphi^\lambda(x)|dx)\\
   &\leq C|\Omega^\lambda|^{\frac{1}{2}}(\int_{\Omega^\lambda}|\varphi^\lambda(x)|^2dx)^{\frac{1}{2}}
   +C|\Omega^\lambda|^{\frac{1}{2}}(\int_{\Omega^\lambda}|\nabla\varphi^\lambda(x)|^2dx)^{\frac{1}{2}}.
   \end{split}
\end{equation}
Since $|\Omega^\lambda|=\lambda/\kappa\rightarrow0$ as $\lambda\rightarrow0^+$, we get from \eqref{3-18}
\begin{equation}\label{3-19}
(\int_{\Omega^\lambda}|\varphi^\lambda(x)|^2dx)^{\frac{1}{2}}\leq C|\Omega^\lambda|^{\frac{1}{2}}(\int_{\Omega^\lambda}|\nabla\varphi^\lambda(x)|^2dx)^{\frac{1}{2}}.
\end{equation}
Combining \eqref{3-17} and \eqref{3-19} we obtain
\begin{equation}\label{3-20}
\int_D(G\omega^\lambda(x)+q(x)-\mu^\lambda)\omega^\lambda(x)dx\geq -C|\Omega^\lambda|(\int_{\Omega^\lambda}|\nabla\varphi^\lambda(x)|^2dx)^{\frac{1}{2}}.
\end{equation}
Notice that by $L^p$ estimate
\begin{equation}\label{3-21}
\begin{split}
\int_{\Omega^\lambda}|\nabla\varphi^\lambda(x)|^2dx&\leq2(\int_{\Omega^\lambda}|\nabla G\omega^\lambda(x)|^2dx+\int_{\Omega^\lambda}|\nabla q(x)|^2dx)\\
&\leq 2(|\nabla G\omega^\lambda|^2_{L^\infty(D)}|\Omega^\lambda|+|\nabla q|_{L^\infty(D)}^2|\Omega^\lambda|)\\
&\leq C\lambda|\nabla G\omega^\lambda|^2_{W^{1,3}(D)}+2\lambda|\nabla q|_{L^\infty(D)}^2|\\
&\leq C\lambda(1+|\omega^\lambda|^2_{L^3(D)}|)\\
&\leq C\lambda(1+\lambda^\frac{2}{3})\\
&\leq C\lambda.
\end{split}
\end{equation}

 \eqref{3-20} and \eqref{3-21} together give the desired result.

\end{proof}

Now we are able to conclude the following crucial estimate for the Lagrange multiplier $\mu^\lambda$.

\begin{lemma}\label{3-100}
\begin{equation}
\min_{\overline{D}}q<\mu^\lambda\leq\min_{\overline{D}}q+C\lambda^{\frac{1}{2}}.
\end{equation}
\end{lemma}
\begin{proof}
It is easy to see that the following identity holds
\begin{equation}\label{3-22}
  E(\omega^\lambda)=-\frac{1}{2}\int_DG\omega^\lambda(x)\omega^\lambda(x)dx+\int_D(G\omega^\lambda(x)+q(x)-\mu^\lambda)\omega^\lambda(x)dx+\lambda\mu^\lambda.
\end{equation}
Since $|G\omega^\lambda|_{L^\infty(D)}\leq C\lambda^\frac{1}{2}$, combining Lemma \ref{3-10} and Lemma \ref{3-15} we obtain
\begin{equation}\label{3-23}
  \mu^\lambda\leq\min_{\overline{D}}q+C\lambda^{\frac{1}{2}}.
\end{equation}
On the other hand, since $\Omega^\lambda$ is not empty, we can choose $x\in\Omega^\lambda$, then
\begin{equation}\label{3-24}
\mu^\lambda>G\omega^\lambda(x)+q(x)\geq q(x)\geq \min_{\overline{D}}q.
\end{equation}
Here we use the fact $G\omega^\lambda\geq0$ in $D$ by the maximum principle.

\end{proof}

Now we are ready to give the limiting behavior of the minimizer $\omega^\lambda$, which is equivalent to the limiting behavior of $\Omega^\lambda$ as $\lambda\rightarrow0^+.$
\begin{lemma}\label{3-200}

\[\lim_{\lambda\rightarrow0^+}\sup_{x\in \Omega^\lambda}|q(x)-\min_{\overline{D}}q|=0.\]

\end{lemma}
\begin{proof}
Notice that \eqref{3-24} holds for any $x\in \Omega^\lambda$. Combining Lemma \ref{3-100} we get the desired result.

\end{proof}

\subsection{Proof of Theorem \ref{1-11}}
Now we are ready to prove Theorem \ref{1-11}.
\begin{proof}[Proof of Theorem \ref{1-11}]
Let $\omega^\lambda$ be the unique minimizer obtained in Proposition \ref{3-3}. First we show that $\omega^\lambda$ is a weak solution of \eqref{1-7}. For any $\xi\in C^{\infty}_c(D)$ and $x\in D$, we consider the following ordinary differential equation
\begin{equation}\label{4-1}
\begin{cases}\frac{d\Phi_t(x)}{dt}=\nabla^\perp\xi(\Phi_t(x)) &t\in\mathbb R, \\
\Phi_0(x)=x.
\end{cases}
\end{equation}
Since $\nabla^\perp\xi$ is a smooth vector field with compact support, $\eqref{4-1}$ has a global solution. It is easy to check that $\nabla^\perp\xi$ is divergence-free, so $\Phi_t$ is an area-preserving transformation from $D$ to $D$, that is, for any measurable set $A\subset D$, we have $|\{\Phi_t(x)\mid x\in A\}|=|A|$. Let $\{\omega_t\}_{t\in\mathbb R}$ be a family of test functions defined by
\begin{equation}
\omega_t(x):=\omega^\lambda(\Phi_t(x)).
\end{equation}
It is obvious that $\omega_t\in \mathcal{M}^\lambda$, so $\frac{dE(\omega_t)}{dt}|_{t=0}=0$.
 Expanding $E(\omega_t)$ at $t=0$ we obtain for $|t|$ small
\[\begin{split}
E(\omega_t)=&\frac{1}{2}\int_D\int_DG(x,y)\omega^\lambda(\Phi_t(x))\omega^\lambda(\Phi_t(y))dxdy+\int_Dq(x)\omega^\lambda(\Phi_t(x))dx\\
=&\frac{1}{2}\int_D\int_DG(\Phi_{-t}(x),\Phi_{-t}(y))\omega^\lambda(x)\omega^\lambda(y)dxdy\int_Dq(\Phi_{-t}(x))\omega^\lambda(x)dx\\
=&E(\omega^\lambda)+t\int_D\omega^\lambda\nabla^\perp(G\omega^\lambda+q)\cdot\nabla\xi dx+o(t).
\end{split}\]
 Therefore we get
\[\int_D\omega^\lambda\nabla^\perp(G\omega^\lambda+q)\cdot\nabla\xi dx=0.\]

Note that \eqref{1-12} has been verified in the construction of $\omega^\lambda$ in Section 3.
Now we prove \eqref{1-14} by contradiction.

Suppose that there exist $\lambda_j>0, x_j\in\Omega^{\lambda_j}$ for $j=1,2,\cdot\cdot\cdot,$ such that $\lambda\rightarrow0^+$ as $j\rightarrow+\infty$ and $x_j\notin \mathcal{S}_{\delta_0}$ for some $\delta_0>0$. By the continuity of $q$, we have
\begin{equation}\label{4-2}
\inf_j{q(x_j)}>\min_{\overline{D}}q.
\end{equation}
On the other hand, by Lemma \ref{3-200} we have
\begin{equation}\label{4-3}
  \lim_{j\rightarrow+\infty}q(x_j)=\min_{\overline{D}}q,
\end{equation}
which is a contradiction.

\end{proof}

\section{Proof of Theorem \ref{1-15}}

To prove Theorem \ref{1-15}, we consider a similar variational problem.

Let $\delta_0$ be chosen in Theorem \ref{1-15}. For $\lambda>0$ sufficiently small, define
\begin{equation}\label{5-1}
\begin{split}
 \mathcal{N}^\lambda:=&\{\omega\in L^\infty(D)\mid \omega=\sum_{p=1}^{k+1}\omega_p, supp(\omega_p)\subset B_{\delta_0}(x_p) \text{ for } p=1,\cdot\cdot\cdot k+l, \\
 &\int_D\omega_i(x)dx=\lambda \text{ and } 0\leq\omega_i\leq\kappa_i \text{ for } i=1,\cdot\cdot\cdot,k,\\
  &\int_D\omega_j(x)dx=-\lambda \text{ and } -\kappa_j\leq\omega_j\leq0 \text{ for } j=k+1,\cdot\cdot\cdot,k+l\}.
 \end{split}
\end{equation}
The energy functional on $\mathcal{N}^\lambda$ is still defined by
\begin{equation}\label{5-2-1}
  E(\omega)=\frac{1}{2}\int_D\int_DG(x,y)\omega(x)\omega(y)dxdy+\int_Dq(x)\omega(x)dx.
\end{equation}
We consider the minimization of $E$ on $\mathcal{N}^\lambda$, that is,
\begin{equation}
c^*_\lambda=\inf\{E(\omega)\mid\omega\in\mathcal{N}^\lambda\}.
\end{equation}

\subsection{Existence of a minimizer}
As in Section 3, we first establish the following result.
\begin{lemma}
$\mathcal{N}^\lambda$ is a sequentially compact set in $L^\infty(D)$.
\end{lemma}

\begin{proof}
Let $\{\omega_n\}$ be a sequence in $L^\infty(D)$ and $\omega_n\rightarrow\omega_0\in L^\infty(D)$ weakly star as $n\rightarrow+\infty$. It suffices to show $\omega_0\in\mathcal{N}^\lambda$. By the definition of weak star convergence it is easy to check that for each $p$, $1\leq p\leq k+l,$  $\omega_n I_{B_{\delta_0}(x_p)}\rightarrow\omega_0I_{B_{\delta_0}(x_p)}$ weakly star in $L^\infty(D)$. Then we repeat the argument in Lemma \ref{3-344} to obtain
\[0\leq\omega_0I_{B_{\delta_0}(x_i)}\leq\kappa_i, \int_D\omega_0I_{B_{\delta_0}(x_i)}(x)dx=\lambda, \text{ for }i=1,\cdot\cdot\cdot,k,\]
\[-\kappa_j\leq\omega_0I_{B_{\delta_0}(x_j)}\leq0, \int_D\omega_0I_{B_{\delta_0}(x_j)}(x)dx=-\lambda, \text{ for }j=1,\cdot\cdot\cdot,k.\]
Therefore $\omega_0=\sum_{p=1}^{k+l}\omega_n I_{B_{\delta_0}(x_p)}\in\mathcal{N}^\lambda.$
\end{proof}

\begin{proposition}\label{5-2}
$c^*_\lambda$ can be achieved by an absolute minimizer $w^\lambda\in \mathcal{N}^\lambda$ with the following the form
\begin{equation}\label{5-3}
  w^\lambda=\sum_{i=1}^k\kappa_iI_{U_i^\lambda}-\sum_{j=k+1}^{k+l}\kappa_jI_{U_j^\lambda},
\end{equation}
where for each $1\leq i\leq k$
\begin{equation}\label{5-4}
U_i^\lambda =\{x\in D\mid Gw^\lambda(x)+q(x)<\nu_i^\lambda\}\cap{B_{\delta_0}(x_i)},\,\,|U_i^\lambda|=\lambda,
\end{equation}
and for each $k+1\leq j\leq k+l$
\begin{equation}\label{5-5}
U_j^\lambda =\{x\in D\mid Gw^\lambda(x)+q(x)>\nu_j^\lambda\}\cap{B_{\delta_0}(x_j)},\,\,|U_j^\lambda|=\lambda.
\end{equation}
Here $\nu_p^\lambda\in\mathbb{R}, 1\leq p\leq k+l$, is the Lagrange multiplier depending on $\lambda$.
\end{proposition}
\begin{proof}
First by $L^p$ estimate and Sobolev embedding
\begin{equation}\label{5-6}
 |G\omega|_{L^\infty(D)}\leq C|G\omega|_{W^{2,2}(D)}\leq C|\omega|_{L^2(D)}\leq C\lambda^{\frac{1}{2}},\,\,\forall \omega\in\mathcal{N}^\lambda.
\end{equation}
Here $C$ still denotes various positive numbers not depending on $\lambda$. Therefore we obtain
\begin{equation}\label{5-7}
|\int_D\int_DG(x,y)\omega(x)\omega(y)dxdy|= |\int_DG\omega(x)\omega(x)dx|\leq C\lambda^{\frac{3}{2}}.
\end{equation}
On the other hand, for any $\omega\in\mathcal{N}^\lambda$ with $\omega=\sum_{p=1}^{k+l}\omega_p$, since $\omega_i\geq0$ for $1\leq i\leq k$ and $\omega_j\leq0$ for $k+1\leq j\leq k+l$, we have
\begin{equation}\label{5-8}
\begin{split}
  \int_Dq(x)\omega(x)dx= \sum_{p=1}^{k+l} \int_Dq(x)\omega_p(x)dx
    \geq \sum_{p=1}^{k+l}\int_Dq(x_p)\omega_p(x)dx=\lambda(\sum_{i=1}^kq(x_i)-\sum_{j=k+1}^{k+l}q(x_j)).
\end{split}
\end{equation}
From \eqref{5-7}\eqref{5-8} we can easily get
\begin{equation}\label{5-9}
\inf_{\mathcal{N}^\lambda}E\geq\lambda(\sum_{i=1}^kq(x_i)-\sum_{j=k+1}^{k+l}q(x_j))-C\lambda^\frac{3}{2},
\end{equation}
which implies that $E$ is bounded from below on $\mathcal{N}^\lambda$. Now we choose a minimizing sequence $\{\omega_n\}\subset\mathcal{N}^\lambda$ such that as $n\rightarrow+\infty$
\[E(\omega_n)\rightarrow c^*_\lambda.\]
Since $\mathcal{N}^\lambda$ is compact in the weak star topology of $L^\infty(D)$ and $E$ is weakly star continuous in $L^\infty(D)$, we deduce that there exists $w^\lambda\in\mathcal{N}^\lambda$ such that
\begin{equation}\label{5-10}
   E(\omega_n)\rightarrow E(w^\lambda)=c^*_\lambda.
\end{equation}

Since $w^\lambda\in\mathcal{N}^\lambda$, we can write $w^\lambda=\sum_{p=1}^{k+l}w^\lambda_p$. Now we show that $w^\lambda$ satisfies \eqref{5-3}. We need to consider the following two different cases.

Case 1: For $1\leq p\leq k$,
\begin{equation}\label{5-11}
   w_p^\lambda=\kappa_p I_{\{x\in D\mid Gw^\lambda(x)+q(x)<\nu^\lambda_p\}\cap B_{\delta_0}(x_p)}
\end{equation}
for some $\nu^\lambda_p\in\mathbb R$. To show this, we define a family of test functions $w_s^\lambda=w^\lambda+s(z_0-z_1), s>0$, where $z_0$ and $z_1$ satisfy
\begin{equation}
\begin{cases}
z_0,z_1\in L^\infty(D),\,\, z_0,z_1\geq 0\text{ a.e. in } D,

 \\ \int_Dz_0(x)dx=\int_D z_1(x)dx,

 \\ supp(z_0),supp(z_1)\subset B_{\delta_0}(x_p),
 \\z_0=0 \quad\text{in } D\setminus\{x\in D\mid w^\lambda(x)\leq\kappa_p-\delta\},
 \\z_1=0 \quad\text{in } D\setminus\{x\in D\mid w^\lambda(x)\geq\delta\}.
\end{cases}
\end{equation}
Here $\delta>0$ is small. It is not hard to check that for fixed $z_0,z_1$ and $\delta$, $\omega_s\in \mathcal{N}^\lambda$ for sufficiently small $s$. Since $w^\lambda$ is a minimizer, we get
\[0\leq\frac{dE(w_s)}{ds}|_{s=0^+}=\int_Dz_0(x)(Gw^\lambda(x)+q(x))dx-\int_Dz_1(x)(Gw^\lambda(x)+q(x))dx.\]
By the choice of $z_0$ and $z_1$ we obtain
\begin{equation}\label{5-12}
\sup_{\{x\in D\mid w^\lambda(x)>0\}\cap B_{\delta_0}(x_p)}(G w^\lambda+q)\leq\inf_{\{x\in D\mid w^\lambda(x)<\kappa_p\}\cap B_{\delta_0}(x_p)}(G w^\lambda+q).
\end{equation}
Since $D$ is simply-connected and $G w^\lambda+q$ is continuous in $\{x\in D\mid w^\lambda(x)>0\}\cap B_{\delta_0}(x_p)$, \eqref{5-12} is in fact an equality, i.e.,
\begin{equation}
\sup_{\{x\in D\mid w^\lambda(x)>0\}\cap B_{\delta_0}(x_p)}(G w^\lambda+q)=\inf_{\{x\in D\mid w^\lambda(x)<\kappa_p\}\cap B_{\delta_0}(x_p)}(G w^\lambda+q):=\nu^\lambda_p.
\end{equation}
Then it is easy to check that
\begin{equation}
\begin{cases}
w^\lambda=0 &\text{ a.e. } \text{in }\{x\in D\mid Gw^\lambda(x)+q(x)\geq\nu_p^\lambda\}\cap B_{\delta_0}(x_p),
 \\ w^\lambda=\kappa_p &\text{ a.e. } \text{in }\{x\in D\mid Gw^\lambda(x)+q(x)<\nu_p^\lambda\}\cap B_{\delta_0}(x_p).
\end{cases}
\end{equation}
So we obtain
\[w_p^\lambda=\kappa_pI_{\{x\in D\mid G\omega^\lambda(x)+q(x)<\nu_p^\lambda\}\cap B_{\delta_0}(x_p)}.\]

Case 2: For $k+1\leq p\leq k+l$,
\begin{equation}\label{5-13}
   w_p^\lambda=-\kappa_p I_{\{x\in D\mid Gw^\lambda(x)+q(x)>\nu^\lambda_p\}\cap B_{\delta_0}(x_p)}
\end{equation}
for some $\nu^\lambda_p\in\mathbb R$. In this case, we choose $w_s^\lambda=w^\lambda+s(z_0-z_1)$ as test function, where $s>0$ and $z_0, z_1$ satisfy
\begin{equation}
\begin{cases}
z_0,z_1\in L^\infty(D),\,\, z_0,z_1\leq 0\text{ a.e. in } D,

 \\ \int_Dz_0(x)dx=\int_D z_1(x)dx,
 \\ supp(z_0),supp(z_1)\subset B_{\delta_0}(x_p),
 \\z_0=0 \quad\text{in } D\setminus\{x\in D\mid w^\lambda(x)\geq \delta-\kappa_p\},
 \\z_1=0 \quad\text{in } D\setminus\{x\in D\mid w^\lambda(x)\leq-\delta\},
\end{cases}
\end{equation}
for $\delta>0$ small. The rest of the proof is similar to Case 1, therefore we omit it.

\end{proof}

\begin{remark}
Following the argument in Lemma \ref{3-1001}, we can also prove that $E$ has only one minimizer.
\end{remark}

\subsection{Limiting behavior of the minimizer as $\lambda\rightarrow0^+$}
Let $w^\lambda$ be the minimizer obtain in the last subsection. The following are several lemmas concerning the limiting behavior as $\lambda\rightarrow0^+$.
\begin{lemma}\label{5-101}
\begin{equation}\label{5-102}
  E(w^\lambda)\leq\lambda(\sum_{i=1}^{k}q(x_i)-\sum_{j=k+1}^{k+l}q(x_j))+C\lambda^\frac{3}{2}.
\end{equation}
\end{lemma}

\begin{proof}
For $\lambda$ sufficiently small, we define a test function
 \[v^\lambda=\sum_{i=1}^k\kappa_i I_{B_{\varepsilon_i}(x_i^\lambda)}-\sum_{j=k+1}^{k+l}\kappa_j I_{B_{\varepsilon_j}(x_j^\lambda)},\]
  where $\varepsilon_p$ satisfies $\kappa_p\pi\varepsilon_p^2=\lambda$, $B_{\varepsilon_p}(x_p^\lambda)\subset B_{\delta_0}(x_p)$ and $|x_p-x^\lambda_p|=\varepsilon_p$ for each $1\leq p\leq k+l$. Note that such test function exists since $D$ satisfies the interior sphere condition. It is obvious that $v^\lambda\in\mathcal{N}^\lambda$. So we have
  \begin{equation}\label{5-103}
   \begin{split}
     E(w^\lambda) \leq E(v^\lambda)
       =\frac{1}{2}\int_D\int_DG(x,y)v^\lambda(x)v^\lambda(y)dxdy+\int_Dq(x)v^\lambda(x)dx.
   \end{split}
\end{equation}
By \eqref{5-7},
\begin{equation}\label{5-104}
|\frac{1}{2}\int_D\int_DG(x,y)v^\lambda(x)v^\lambda(y)dxdy|\leq C\lambda^\frac{3}{2}.
\end{equation}
On the other hand,
\begin{equation}\label{5-105}
  \begin{split}
    \int_Dq(x)v^\lambda(x)dx &= \sum_{i=1}^k\kappa_i \int_{B_{\varepsilon_i}(x_i^\lambda)}q(x)dx-\sum_{j=k+1}^{k+l}\kappa_j \int_{B_{\varepsilon_j}(x_j^\lambda)}q(x)dx\\
    &=\sum_{i=1}^k\kappa_i \int_{B_{\varepsilon_i}(x_i^\lambda)}(q(x)-q(x_i))dx-\sum_{j=k+1}^{k+l}\kappa_j \int_{B_{\varepsilon_j}(x_j^\lambda)}(q(x)-q(x_j))dx\\
    &+\sum_{i=1}^k\kappa_i \int_{B_{\varepsilon_i}(x_i^\lambda)}q(x_i)dx-\sum_{j=k+1}^{k+l}\kappa_j \int_{B_{\varepsilon_j}(x_j^\lambda)}q(x_j)dx\\
     &\leq\sum_{i=1}^k\kappa_i \int_{B_{\varepsilon_i}(x_i^\lambda)}|\nabla q|_{L^\infty(D)}|x-x_i|dx+\sum_{j=k+1}^{k+l}\kappa_j \int_{B_{\varepsilon_j}(x_j^\lambda)}|\nabla q|_{L^\infty(D)}|x-x_j|dx\\
    &+\lambda(\sum_{i=1}^k q(x_i)-\sum_{j=k+1}^{k+l} q(x_j))\\
     &\leq\sum_{p=1}^{k+l}\kappa_p \int_{B_{\varepsilon_p}(x_p^\lambda)}C\varepsilon_pdx
   +\lambda(\sum_{i=1}^kq(x_i)-\sum_{j=k+1}^{k+l} q(x_j))\\
   &\leq C\lambda^\frac{3}{2}
   +\lambda(\sum_{i=1}^k q(x_i)-\sum_{j=k+1}^{k+l} q(x_j)).
  \end{split}
\end{equation}

Combining \eqref{5-103},\eqref{5-104} and \eqref{5-105} we get the desired result.
\end{proof}

\begin{lemma}\label{5-106}
For each $p$, $1\leq p\leq k+l, $ we have
\begin{equation}\label{5-107}
\int_D(Gw^\lambda(x)+q(x)-\nu_p^\lambda)w_p^\lambda(x)dx\geq -C\lambda^\frac{3}{2}.
\end{equation}

\end{lemma}

\begin{proof}
We only prove the case $1\leq p\leq k$, for the other part the proof is similar. For simplicity we denote $\varphi^\lambda=Gw^\lambda+q-\nu^\lambda_p, \varphi^\lambda_-=\min{\{\varphi^\lambda,0\}}$ Then by H\"older's inequality
\begin{equation}\label{5-108}
\begin{split}
 \int_D\varphi^\lambda(x)w^\lambda_p(x)dx
 =\kappa_p\int_{U^\lambda_p}\varphi^\lambda(x)dx
 \geq-\kappa_p|U^\lambda_p|^\frac{1}{2}(\int_{U^\lambda_p}|\varphi^\lambda(x)|^2dx)^\frac{1}{2}.
 \end{split}
\end{equation}
On the other hand, by Sobolev embedding $W^{1,1}(B_{\delta_0}(x_p))\hookrightarrow L^2(B_{\delta_0}(x_p))$ and H\"older's inequality,
\begin{equation}\label{5-109}
\begin{split}
(\int_{U^\lambda_p}|\varphi^\lambda(x)|^2dx)^\frac{1}{2}=&(\int_{B_{\delta_0}(x_p)}|\varphi^\lambda_-(x)|^2dx)^\frac{1}{2}\\
\leq& C(\int_{B_{\delta_0}(x_p)}|\varphi^\lambda_-(x)|dx+\int_{B_{\delta_0}(x_p)}|\nabla\varphi^\lambda_-(x)|dx)\\
\leq&C|U^\lambda_p|^\frac{1}{2}(\int_{U^\lambda_p}|\varphi^\lambda(x)|^2dx)^\frac{1}{2}
+C|U^\lambda_p|^\frac{1}{2}(\int_{U^\lambda_p}|\nabla\varphi^\lambda(x)|^2dx)^\frac{1}{2}.
\end{split}
\end{equation}
Since $|U^\lambda_p|\rightarrow0$ as $\lambda\rightarrow0^+$, we get from \eqref{5-109}
\begin{equation}\label{5-110}
(\int_{U^\lambda_p}|\varphi^\lambda(x)|^2dx)^\frac{1}{2}\leq C|U^\lambda_p|^\frac{1}{2}(\int_{U^\lambda_p}|\nabla\varphi^\lambda(x)|^2dx)^\frac{1}{2}.
\end{equation}
Taking into account \eqref{5-108} and \eqref{5-110} and using $L^p$ estimate, we obtain
\begin{equation}\label{5-111}
\begin{split}
  \int_D\varphi^\lambda(x)w^\lambda_p(x)dx&\geq -C\lambda (\int_{U^\lambda_p}|\nabla\varphi^\lambda(x)|^2dx)^\frac{1}{2}\\
  &\geq-C\lambda (\int_{U^\lambda_p}|\nabla Gw^\lambda(x)|^2+|\nabla q(x)|^2dx)^\frac{1}{2}\\
  &\geq -C\lambda^\frac{3}{2}(|\nabla Gw^\lambda|_{L^\infty(D)}+|\nabla q|_{L^\infty(D})\\
  &\geq -C\lambda^\frac{3}{2},
  \end{split}
\end{equation}
which completes the proof.
\end{proof}

\begin{lemma}\label{5-200}
(i)For $1\leq i\leq k$,
\begin{equation}\label{5-301}
\nu^\lambda_i>q(x_i)-C\lambda^\frac{1}{2};
\end{equation}
(ii) For  $k+1\leq j\leq k+l$,
\begin{equation}\label{5-302}
\nu^\lambda_j< q(x_j)+C\lambda^\frac{1}{2}.
\end{equation}
\end{lemma}

\begin{proof}
First recall that by $L^p$ estimate $|Gw^\lambda|_{L^\infty(D)}\leq C\lambda^\frac{1}{2}$.
Since $U^\lambda_p$ is not empty, we can choose $y_p\in U^\lambda_p$, then for $1\leq p\leq k$
\begin{equation}\label{5-115-2}
Gw^\lambda(y_p)+q(y_p)<\nu^\lambda_p,
\end{equation}
and for $k+1\leq p\leq k+l$
\begin{equation}\label{5-116}
Gw^\lambda(y_p)+q(y_p)>\nu^\lambda_p.
\end{equation}
From \eqref{5-115-2} we get for $1\leq p\leq k$
\begin{equation}\label{5-117}
   q(x_p)-C\lambda^\frac{1}{2}\leq Gw^\lambda(y_p)+q(y_p)<\nu^\lambda_p.
\end{equation}
From \eqref{5-116} we get for $1\leq p\leq k$
\begin{equation}\label{5-118}
   q(x_p)+C\lambda^\frac{1}{2}\geq Gw^\lambda(y_p)+q(y_p)>\nu^\lambda_p.
\end{equation}
\end{proof}

\begin{lemma}\label{5-112}
(i)For $1\leq i\leq k$,
\begin{equation}\label{5-303}
\nu^\lambda_i\leq q(x_i)+C\lambda^\frac{1}{2};
\end{equation}
(ii) For  $k+1\leq j\leq k+l$,
\begin{equation}\label{5-304}
\nu^\lambda_j\geq q(x_j)-C\lambda^\frac{1}{2}.
\end{equation}
\end{lemma}
\begin{proof}
Notice that $E(w^\lambda)$ can be written as
\begin{equation}\label{5-113}
\begin{split}
E(w^\lambda)&=-\frac{1}{2}\int_D\int_DG(x,y)w^\lambda(x)w^\lambda(y)dxdy+\sum_{p=1}^{k+1}\int_D(Gw^\lambda(x)+q(x)-\nu^\lambda_p)w^\lambda_p(x)dx\\
&+\lambda(\sum_{i=1}^{k}\nu^\lambda_i-\sum_{j=k+1}^{k+l}\nu^\lambda_j).\\
\end{split}
\end{equation}
Taking into account Lemma \ref{5-101} and Lemma \ref{5-106}, we deduce from \eqref{5-113} that
\begin{equation}\label{5-114}
   \sum_{i=1}^{k}\nu^\lambda_i-\sum_{j=k+1}^{k+l}\nu^\lambda_j\leq \sum_{i=1}^{k}q(x_i)-\sum_{j=k+1}^{k+l}q(x_j)+C\lambda^\frac{1}{2}.
\end{equation}
Now \eqref{5-301}, \eqref{5-302} and \eqref{5-114} together give the desired result.

\end{proof}

\begin{lemma}\label{3-500}
For each $p$, $1\leq p\leq k+l,$ we have
\begin{equation}\label{5-115}
  \lim_{\lambda\rightarrow0^+}\sup_{x\in U^\lambda_p}|q(x)-q(x_p)|=0.
\end{equation}
\end{lemma}

\begin{proof}
First by Lemma \ref{5-112} and the definition of $U^\lambda_p$, for each $x\in U^\lambda_p$, we have
\[q(x_p)-C\lambda^\frac{1}{2}\leq Gw^\lambda(x)+q(x)<\nu^\lambda_p\leq q(x_p)+C\lambda^\frac{1}{2},\,\,1\leq p\leq k,\]
\[q(x_p)-C\lambda^\frac{1}{2}\leq\nu^\lambda< Gw^\lambda(x)+q(x)\leq q(x_p)+C\lambda^\frac{1}{2},\,\,1\leq p\leq k,\]
from which it is easy to see that for each $p$, $1\leq p\leq k+l$ and $x\in U^\lambda_p$,
\[|q(x)-q(x_p)|\leq |Gw^\lambda(x)+q(x)-q(x_p)|+|Gw^\lambda(x)|\leq C\lambda^\frac{1}{2},\]
which implies \eqref{5-115}.
\end{proof}

\subsection{Proof of Theorem \ref{1-15}}
Now we are ready to prove Theorem \ref{1-15}.

\begin{proof}[Proof of Theorem \ref{1-15}]
First we show that for each $p, 1\leq p\leq k+l$, $U^\lambda_p$ shrinks to $x_p$ as $\lambda\rightarrow0^+.$ Suppose that there exist $\delta_1>0, \lambda_n\rightarrow0^+, y_n\in U^\lambda_p$ such that $|y_n-x_p|\geq \delta_0.$ By the continuity of $q$ we deduce that $\inf_{n}|q(y_n)-q(x_p)|>0$, which is a contradiction to Lemma \ref{3-500}.

Now we show that $w^\lambda$ is a weak solution of \eqref{1-7} for $\lambda$ sufficiently small. For any $\xi\in C^{\infty}_c(D)$, let $\Phi_t(x)$ be defined by \eqref{4-1}. Then since $U^\lambda_p$ shrinks to $x_p$ for each $p$, we deduce that for $|t|<<1$, $w_t:=w^\lambda(\Phi_t(\cdot))\in \mathcal{N}^\lambda$, so we still have  $\frac{dE(w_t)}{dt}|_{t=0}=0$, which gives
\[\int_Dw^\lambda\nabla^\perp(Gw^\lambda+q)\cdot\nabla\xi dx=0.\]
\end{proof}

\section{Further Discussion}

When we consider vortex patch solutions near a harmonic function, there are four possibilities:
\begin{enumerate}
\item a positive vortex patch near a minimum point of the harmonic function;
\item a negative vortex patch near a minimum point of the harmonic function;
\item a negative vortex patch near a maximum point of the harmonic function;
\item a positive vortex patch near a maximum point of the harmonic function.
\end{enumerate}
As has been mentioned in Remark \ref{2-102}, $(1)$ is equivalent to $(3)$ and $(2)$ is equivalent to $(4)$. Now we give an example in one dimension to illustrate $(1)$ and $(2)$.

Let $u$ be harmonic in $(0,1)\subset\mathbb R$ with boundary condition $u(0)=0, u(1)=1.$  We perturb $u$ near $x=0$ by constructing two kinds of vortex patch solutions with the same boundary condition, that is, we consider the following two problems:
\begin{equation}\label{6-1}
\mathcal{P}_1:
\begin{cases}
-\frac{d^2u^\lambda}{dx^2}=I_{\{u^\lambda<\mu^\lambda\}} &x\in(0,1),\\
|\{u^\lambda<\mu^\lambda\}|=\lambda,\\
u^\lambda(0)=u &\text{ on } \partial(0,1),
\end{cases}
\end{equation}
\begin{equation}\label{6-2}
\mathcal{P}_2:
\begin{cases}
-\frac{d^2v^\lambda}{dx^2}=-I_{\{v^\lambda>\nu^\lambda\}} &x\in(0,1),\\
|\{v^\lambda>\nu^\lambda\}|=\lambda,\\
v^\lambda(0)=u &\text{ on } \partial(0,1).
\end{cases}
\end{equation}
Explicit solutions to $\mathcal{P}_1$ and $\mathcal{P}_2$ are
\begin{equation}\label{6-3}
  u^\lambda(x)=
  \begin{cases}
  -x^2+(2\lambda+1-\lambda^2)x &0\leq x\leq\lambda,\\
  (1-\lambda^2)x+\lambda^2 &\lambda\leq x\leq 1,
  \end{cases}
\end{equation}
\begin{equation}\label{6-4}
  v^\lambda(x)=
  \begin{cases}
  x^2+(\lambda-1)^2x &0\leq x\leq\lambda,\\
  (\lambda^2+1)x-\lambda^2 &\lambda\leq x\leq 1.
  \end{cases}
\end{equation}

In Theorem \ref{1-11} and \ref{1-15}, we have solved $\mathcal{P}_1$ in two dimensions. In this section we consider $\mathcal{P}_2$.
More precisely, we construct steady vortex patches near maximum points of $q$ with positive vorticity and near minimum points of $q$ with negative vorticity.

The results are as follows.
\begin{theorem}\label{1-11-2}
Let $q\in C^2(D)\cap C^1(\overline{D})$ be a harmonic function and $\kappa$ be a positive real number. Set $\mathcal{G}:=\{x\in \overline{D}\mid q(x)=\max_{ \overline{D}}q\}$. Then
for any given positive number $\lambda$ with $\lambda<\kappa|D|$, there exists a weak solution $\omega^\lambda$ of \eqref{1-7} having the form
\begin{equation}\label{1-12-2}
  \omega^\lambda=\kappa I_{\Omega^\lambda}, \,\, \Omega^\lambda=\{x\in D\mid G\omega^\lambda(x)+q(x)>\mu^\lambda\},\,\,\kappa|\Omega^\lambda|=\lambda
\end{equation}
for some $\mu^\lambda\in\mathbb{R}$ depending on $\lambda$.
Furthermore, if $q$ is not a constant, then $\mathcal{G}\subset\partial D$ and $\Omega^\lambda$ approaches $\mathcal{G}$, or equivalently, for any $\delta>0$, there exists $\lambda_0>0$, such that for any $\lambda<\lambda_0$, we have
\begin{equation}\label{1-14-2}
\omega^\lambda\subset\mathcal{G}_\delta:=\{x\in D\mid dist(x,\mathcal{G})<\delta\}.
\end{equation}
\end{theorem}

\begin{theorem}\label{1-15-2}
Let $q\in C^2(D)\cap C^1(\overline{D})$ be a harmonic function, $k,l$ be two nonnegative integers and $\kappa_1,\cdot\cdot\cdot,\kappa_{k+l}$ be $k+l$ positive real numbers. Suppose that $\{x_1,x_2,\cdot\cdot\cdot,x_k\}\subset\partial D$ are $k$ different strict local minimum points of $q$ on $\overline{D}$, and $\{x_{k+1},x_{k+2},\cdot\cdot\cdot,x_{k+l}\}\subset\partial D$ are $l$ different strict local maximum points of $q$ on $\overline{D}$.
Then there exists a $\lambda_0>0$, such that for any $0<\lambda<\lambda_0$, there exists a weak solution of \eqref{1-7} $w^\lambda$ having the form
\begin{equation}\label{1-16-2}
  w^\lambda=-\sum_{i=1}^k\kappa_iI_{U_i^\lambda}+\sum_{j=k+1}^{k+l}\kappa_jI_{U_j^\lambda},
\end{equation}
where for each $1\leq i\leq k$
\begin{equation}\label{5-4-2}
U_i^\lambda =\{x\in D\mid Gw^\lambda(x)+q(x)<\nu_i^\lambda\}\cap{B_{\delta_0}(x_i)},\,\,|U_i^\lambda|=\lambda,
\end{equation}
and for each $k+1\leq j\leq k+l$
\begin{equation}\label{5-5-2}
U_j^\lambda =\{x\in D\mid Gw^\lambda(x)+q(x)>\nu_j^\lambda\}\cap{B_{\delta_0}(x_j)},\,\,|U_j^\lambda|=\lambda,
\end{equation}
where $\nu_p^\lambda\in\mathbb{R},$ for $1\leq p\leq k+l$, is the Lagrange multiplier depending on $\lambda$.
Here $\delta_0$ is chosen to be sufficiently small such that $x_i$ is the unique minimum point of $q$ on $\overline{B_{\delta_0}(x_i)\cap D}$ for $i=1,\cdot\cdot\cdot,k$, $x_i$ is the unique maximum point of $q$ on $\overline{B_{\delta_0}(x_j)\cap D}$ for $j=k+1,\cdot\cdot\cdot,k+l$, and $\overline{B_{\delta_0}(x_{p_1})\cap D}\cap\overline{B_{\delta_0}(x_{p_2})\cap D}=\varnothing$ for $1\leq p_1,p_2\leq k+l, p_1\neq p_2.$
Moreover, $U_p^\lambda$ shrinks to $x_p$ for each $1\leq p\leq k+l$ as $\lambda\rightarrow0^+$, or equivalently, for any $\delta>0$, there exists a $\lambda_0>0$, such that for any $\lambda<\lambda_0$, we have
\begin{equation}\label{1-17-2}
U^\lambda_p\subset B_{\delta}(x_p)\cap D.
\end{equation}

\end{theorem}

Proofs for Theorem \ref{1-11-2} and Theorem \ref{1-15-2} are similar to those for Theorem \ref{1-11} and Theorem \ref{1-15}. More specifically, we consider the maximization of $E$ on $\mathcal{M}^\lambda$ and $\mathcal{N}^\lambda$ respectively and then analyze the limiting behavior of the maximizers as $\lambda\rightarrow0^+$. Since the details are almost the same, we omit it here.

~\\
\noindent{\bf Acknowledgments:}
 Daomin Cao was partially supported by Hua Luo Geng Center of Mathematics,
AMSS, CAS, he was also partially supported by NNSF of China grant No.11331010 and No.11771469.
Guodong Wang was supported by NNSF of China grant No.11771469.

\end{document}